\documentclass[a4paper,12pt]{article}
\usepackage{amsmath,amsfonts,amsthm,ifsym}
\newtheorem{theorem}{Theorem}[section]
\newtheorem{proposition}[theorem]{Proposition}
\newtheorem{remark}[theorem]{Remark}
\newtheorem{lemma}[theorem]{Lemma}
\begin{document}

\title{KAM tori for the generalized Bejamin-Bona-Mahony equation}
\author{{Guanghua Shi$^a$,  Dongfeng Yan$^b$\footnote{Corresponding author. Email:yand11@fudan.edu.cn}}\\
{\em\small $^a$ College  of Mathematics and Statistics, Hunan Normal University,}\\
{\em\small Changsha, Hunan  410006, China}\\
{\em\small $^b$ School of Mathematics and Statistics, Zhengzhou University,}\\
{\em\small Zhengzhou, Henan, 450001, China}\\
}

\maketitle
\begin{abstract}
A generalized Benjamin-Bona-Mahony (gBBM) equation
\begin{align*}
u_t-u_{xxt}+u_x+u^4u_x=0,
\end{align*}
subject to the periodic boundary condition is studied in this paper. Based on a new infinite dimensional Kolomogorov-Arnold-Moser (KAM) theorem with normal frequencies of finite limit-points, it is shown that the gBBM equation admits plenty of time-quasi-periodic solutions with two frequencies of high modes.
\end{abstract}

\section{Introduction and main result}

The Benjamin-Bona-Mahony (BBM) equation
\begin{align}\label{c00}
u_t-u_{xxt}+u_x+uu_x=0,
\end{align}
which is also known as the regularized long-wave equation (RLWE), is initially proposed by Benjamin, Bona and Mahony in 1972 as an improvement of the KdV equation for modeling the long gravity waves of small amplitude in nonlinear dispersive systems \cite{Benjamin1972}. In \cite{Benjamin1972},  Benjamin and his partners established the existence and uniqueness of the global solutions to the intial value problem of the BBM equation. The existence and uniqueness of the periodic solutions for the BBM equation is studied in \cite{Medeiros1977}.
Besides, it is worthwhile to mention that the BBM equation admits only three conservation laws \cite{Olver1979}, while the KdV equation possesses infinite conservation laws.
Subsequently, many research work for the BBM equation were conducted. For instance, the global existence of the BBM equation in arbitrary dimension was established in Ref. \cite{Avrin1985}, while the long time dynamics of the BBM equation was investigated in Ref. \cite{Wang2014}.

We mention that none of the aforementioned papers study the time-quasi-periodic solutions for the BBM equation, in consideration of this, Yuan initially establishes the time-quasi-periodic solutions for the BBM equation (\ref{c00}) based on his new infinite dimensional KAM theorem with normal frequencies of finite limit-points \cite{Yuan2018}. We shall briefly recall the history of the applications of the KAM theory into partial differential equations (PDEs), the key idea of which is to transform the PDEs into an infinite dimensional Hamiltonian system, then the Hamiltonian system is brought into a normal form with an invariant torus by generating iteratively a sequence of symplectic transformations. The pioneering work to study the PDEs in the frame of KAM theory were carried out by Kuksin \cite{Kuksin1989,Kuksin1993,Kuksin2000,Kuksin2004} and Wayne \cite{Wayne1990}, see also P\"oschel \cite{Poschel1996}. These aforemetioned papers concerned with the PDEs of spatial dimension $d=1,$ for the case when the spatial dimension $d\ge 2,$ Bourgain \cite{Bourgain1994,Bourgain1997,Bourgain1998,Bourgain2004,Bourgain2005} managed to develope a new theory which is based on the Newton iteration, Fr\"ohlich-Spencer technique and the semi-algebraic set to study the KAM tori for the PDEs in high spatial dimension, see also Eliasson-Kuksin \cite{Eliasson2010} and Eliasson-Grebert-Kuksin \cite{Eliasson2016}. The existence of KAM tori and quasi-periodic solutions for PDEs with unbounded nonlinearities was developed by Kuksin \cite{Kuksin1998-1}
for the KdV equation (see also \cite{Kappler-Poschel2003}). Later on, based on the improvement of the Kuksin's Lemma (see \cite{Kappler-Poschel2003}), Liu-Yuan \cite{Liu2010,Liu2011}
extended the unbounded KAM theorem to the limiting case, and established the existence of the quasi-periodic solutions for the derivative nonlinear Schr\"odinger equation and the Benjamin-Ono equation. Recently, the existence of the quasi-peridic solutions for the quasi-linear PDEs was studied in \cite{Baldi2014,Baldi2016,Feola2015}.

It should be pointed out that one of the underlying assumptions of the aforementioned KAM theorem is that the normal frequencies $\lambda_j$'s in the normal form part always cluster to infinity. To be specific, there exists some $\kappa\ge 1$ fulfilling 
$$\lambda_j\approx|j|^{\kappa}\to\infty,\quad\text{as}\quad |j|\to\infty.$$
For example, when one considers the nonlinear Schr\"odinger equation
\begin{align}
iu_t-u_{xx}+mu+u^3=0,\quad u(x,0)=u(x,2\pi)=0,
\end{align}
where $m>0$ is a constant, its normal frequencies take the form $\lambda_j=j^2+m\approx j^2$. It is well known that when one constructs the lower dimensional KAM tori, he has to make sure that the first Melnikov's conditions always hold ture, that is, 
$$\Delta_{kj}=\langle k,\omega\rangle+\lambda_j\ne 0,\forall (k,j)\in\mathbb{Z}^{N+1},$$
in which $\omega$ represents the tangential frequency of the integrable normal form part.
By the analyticity of the perturbation term, one could impose some restriction on $k$, say, 
$$|k|\le K=K_{\nu}\approx 2^{\nu},$$
in which $\nu$ represents the step number of the KAM iteration. Specially, when $|j|>CK$ with $C$ larger enough than the norm of the tangential frequency $|\omega|$,
one has
$$|\Delta_{kj}|>|\lambda_j|-|k||\omega|>|j|^2-|\omega|K>1,$$
which simply indicates that the number of the small divisor $\Delta_{kj}$ is finite in the first Melnikov's conditions. However, when one turns to consider the BBM equation,
the normal frequency takes the following form  
$$\lambda_j=\frac{j}{1+j^2}\to 0,\quad\text{as}\quad j\to\infty,$$ 
which means the number of the small divisors $\Delta_{kj}$ is infinite in the first Melnikov's conditions. This is a huge difference between the former KAM theorems and Yuan's KAM theorem with normal frequencies clustering to finite point \cite{Yuan2018}.

This paper is concerned with the generalized Benjamin-Bona-Mahony (gBBM) equation subject to the periodic boundary condition
\begin{align}\label{c01}
u_t-u_{xxt}+u_x+u^4 u_x=0,\quad u(t,0)=u(t,2\pi).
\end{align}
We mention that the decay property and the long-time behavior of the solutions for the generalized BBM equation were investigated in \cite{Albert1989,Biler1992,Fang2008,Naher2012}. The aim of this paper is to study the quasi-periodic solutions for the gBBM equation (\ref{c01}) by means of the KAM method. 

Due to the new infinite dimensional KAM theorem in \cite{Yuan2018}, our main result of the present paper is presented below.
\begin{theorem}\label{t1}
For given $1\ll n_1\ll n_2,$ in the neighborhood of the zero solution $u=0$, the gBBM equation (\ref{c01}) admits plenty of smooth solutions which are quasi-periodic in time, linear stable and of zero
Lyapunov exponent. More precisely, there exists $\epsilon^*=\epsilon^*(n_1,n_2)>0$ such that, for any $\epsilon\in(0,\epsilon^*)$, there is a subset $\tilde{\mathcal{O}}$ of the initial value set $\mathcal{O}_*:=[\epsilon^{\frac{1}{2}},4\epsilon^{\frac{1}{2}}]^2$ with
$$\text{Leb}{\tilde{\mathcal{O}}}=\text{Leb}\mathcal{O}_*\cdot(1-C\frac{1}{|\ln\epsilon|})$$
and for each $\xi\in\tilde{\mathcal{O}}$, the gBBM equation (\ref{c01}) possesses a quasi-periodic solution $u(x,t)$ of frequency $\omega\in\mathbb{R}^2$ in time t, that is,
$$u(t,x)=\sum\limits_{k\in\mathbb{Z}^2,j\in\mathbb{Z}\setminus\{0\}}\hat{u}(k,j)e^{i(k,\omega)t}e^{ijx}$$
fulfilling
\begin{align*}
\begin{split}
|\omega-\omega_0|&\le C\epsilon^{\frac{1}{2}},\quad \omega_0=(\frac{n_1}{1+n_1^2},\frac{n_2}{1+n_2^2}),\\
|\hat{u}(e_l,n_l)-\xi_l|&<C\epsilon^{\frac{5}{8}},\quad l=1,2,\\
\end{split}\end{align*}
and
$$\sum\limits_{(k,j)\not\in\mathcal{q}}|\hat{u}(k,j)|^2 e^{|k|s_0+2a|j|}|j|^{2p}<C\epsilon^{\frac{5}{8}},\quad\mathcal{q}=(e_l,n_l), l=1,2,$$
where $e_l$ denotes the $l-$th unit vector of $\mathbb{Z}^2$, and $s_0>0$, $a>0$, $p>\frac{1}{2}$ are some constants.
\end{theorem}
\begin{remark}
In the Section 9 of \cite{Yuan2018}, Yuan derives the existence of the $N$-dimensional KAM tori thus quasi-periodic solutions for the BBM equation (\ref{c00}) subject to the periodic boundary condition based on an additional assumption. More precisely, if the spatial varialbe x lies in a closed internal $[0,A]$, then $\frac{2\pi}{A}$ is supposed to be a transcendental number, which is necessary to ensure that the combinations of the frequencies do not vanish and the initial Hamiltonian function can be transformed into the required partial Birkhoff normal form. In the present paper,  we try to drop this assumption, that is, we choose the period to be the ordinary one, say, $A=2\pi$. By choosing suitable admissable set and conducting the partial Birkhoff normal of order 14, we truely obtain the time-quasi-periodic solutions for the gBBM equation (\ref{c01}) with two frequencies of high modes. Unfortunately, we fail to obtain the $N$-dimensional KAM tori and corresponding quasi-periodic solutions.
\end{remark}
\begin{remark}
It should be pointed out that the combinations of the frequencies vanish in many situations when we consider the case $A=2\pi$, say, if $(j_1,j_2,j_3,j_4,j_5,j_6)=(1,-2,-2,3,n,-n)$, then $\lambda_{j_1}+\lambda_{j_2}+\lambda_{j_3}+\lambda_{j_4}+\lambda_{j_5}+\lambda_{j_6}=0.$ So we can not remove the nonresonant terms $z_1z_{-2}z_{-2}z_3z_{n}z_{-n}$ when deriving the Birkhoff normal form. This is the reason why we have to choose some special high modes as the tangential frequencies.
\end{remark}

Let us make some comments. To apply the KAM theorem developed by \cite{Yuan2018}, we firstly turn the gBBM equation (\ref{c01}) into an infinite dimensional Hamiltonian system (\ref{c05}), (\ref{c06}) and (\ref{c07}). Note that the eigenvalue $\lambda_j$ converges to zero, not $+\infty$, which is very different from that of the classical KAM theory. It is worthwhile to point out that since the gBBM equation (\ref{c01}) does not contain any external parameters, it requires to extract the internal parameters $\xi$ from the partial Birkhoff normal form to apply the KAM theorem. To be more specific, we choose $\{z_{\pm n_1},z_{\pm n_2}\}$ as the suitable tangential variables, and the others $\{z_j\}_{j\ne\pm n_1,\pm n_2}$ as the normal ones. It requires to remove those terms with $(j_1,j_2,\cdots,j_6)\in(\Delta_0\cup\Delta_1\cup\Delta_2)\setminus\mathcal{N}$ through a symplectic transformation to obtain the partial Birkhoff normal form of order 6 (for the precise definitions of $\Delta_0,$$\Delta_1$,$\Delta_2$ and $\mathcal{N}$, see Section 3). However, the perturbation term $P=\hat{G}+R$ in the new Hamiltonian (\ref{c12}) fails to satisfy the Assumption C of the KAM theorem in \cite{Yuan2018}, where the size of the Hamiltonian vector field $X_P$ is assumed to be $\epsilon$ and the size of $\partial_\xi X_P$ is assumed to be $\epsilon^{\frac{1}{2}}$. 

We shall make some explanations. Roughly speaking, the perturbation terms satisfy (see Proposition \ref{thm1})
\begin{align*}
R=O(\|z\|^{10}),\quad \hat{G}=O(\|z\|^3\|\hat{z}\|^3),
\end{align*} where $\hat{z}=(z_j)_{j\in\mathbb{N}\setminus\{n_1,n_2\}}.$ Assume that $\|z\|=O(\epsilon^{\frac{1}{8}})$ and $\|\hat{z}\|=O(\epsilon^b)$ with $b>\frac{1}{4}$, then the parameters $\xi$ defined by (\ref{y14}) are of $O(\epsilon^{\frac{1}{2}})$, and the estimations of the corresponding Hamiltonian vector field are given by
\begin{align*}
\|X_{\hat{G}}\|=O(\epsilon^{\frac{3}{8}+b}),\quad \|X_R\|=O(\epsilon^{\frac{5}{4}-2b}).
\end{align*}
Since $b>\frac{1}{4}$, the size of $\|X_{R}\|$ is larger than $\epsilon.$ Therefore, we have to remove some "bad" terms of order 10 with $(j_1,\cdots,j_{10})\in(\Delta_0'\cup\Delta_1')\setminus\mathcal{N}'$ in $R$ (the definitions of $\Delta_0'$,$\Delta_1'$,$\mathcal{N}'$ are given in Section 3). By contrast, in \cite{Gao2012}, the authors only need to eliminate the terms of order 10 with $(j_1,\cdots,j_{10})\in\Delta_0'\setminus\mathcal{N}'$, which is sufficient to check the KAM theorem developed by \cite{Poschel1996}. By the same procedure, one obtains the partial Birkhoff normal form of order 10 in Proposition \ref{thm2}, in view of (\ref{y3}) and (\ref{y4}), we have
\begin{align*}
\|X_{\hat{R}}\|=O(\epsilon),\quad \|X_T\|=O(\epsilon^{\frac{7}{4}-2b})
\end{align*}
If we expect the size of $X_T$ to be less than $\epsilon$, then one has to impose an restriction on $b$, that is, $\frac{7}{4}-2b\ge 1,$ i.e., $b\le \frac{3}{8}$. Unfortunately, under this circumstance, the size of the vector field $X_{\hat{G}}$ is still larger than  $\epsilon,$  which determines that one still needs to 
eliminate the relative large term in $T$ of order 14 with $(j_1,\cdots,j_{14})\in\Delta_0''\setminus\mathcal{N}''$ (the precise definitions of $\Delta_0''$ and $\mathcal{N}''$ are presented in Section 3). Luckily, we remark that the perturbation term appears in the new partial Birkhoff normal form of order 14 fulfills all the assumptions of the KAM theorem in \cite{Yuan2018}, see Propostion \ref{thm3}. Moreover, it is worth mentioning that in the above procedures, one has to guarantee the small divisors 
$|\lambda_{j_1}+\cdots+\lambda_{j_l}|(l=6,10,14)$ do not vanish when $(j_1,\cdots,j_l)$ belongs to some admissible index sets, see Proposition \ref{p3.1}, Lemma \ref{c22-00} and Lemma \ref{y8}.

\section{The Hamiltonian}

After equipped with the symplectic structure $-(1-\partial_{xx})^{-1}\partial_x$ and the model space
$$u\in\mathcal{H}_0^{p_0}:=\{u\in\mathcal{H}^{p_0}(\mathbb{T}:\mathbb{R})|\int_{0}^{2\pi}udx=0\},$$
in which $\mathcal{H}^{p_0}$ is the usual Sobolev space with some $p_0>0$, the gBBM equation (\ref{c01}) can be rewritten as a Hamiltonian system
\begin{align}\label{c02}
u_t=-(1-\partial_{xx})^{-1}\partial_x\nabla_u H(u)
\end{align}
with the Hamiltonian function
\begin{align}\label{c03}
H(u)=\frac{1}{2}\int_{0}^{2\pi} u^2 dx+\frac{1}{30}\int_{0}^{2\pi}u^6dx.
\end{align}

To formulate the statement we need some definitions.
Denote 
$$\bar{\mathbb{Z}}=\mathbb{Z}\setminus\{0\},\quad h_{p_0}:=\{z=(z_j\in\mathbb{C})|\|z\|_p^2=\sum\limits_{j\in\bar{\mathbb{Z}}}|z_j|^2 j^{2p_0}<\infty\}.$$
Define the Fourier transformation $\mathcal{F}:u\to z=(z_j\in\mathbb{C}:j\in\bar{\mathbb{Z}})$ as follows
\begin{align}\label{c04}
u=\sum\limits_{j\in\bar{\mathbb{Z}}}\delta_j z_j\phi_j,\quad \phi_j=\frac{1}{\sqrt{2\pi}}e^{ijx},\quad \delta_j=\sqrt{\frac{|j|}{1+j^2}}.
\end{align}
It is worth noting that $\bar{z}_j=z_{-j}$ iff $u\in\mathbb{R},$ hence $\mathcal{F}$ is isometry from $\mathcal{H}^{p_0}$ to $h_{p}$ with $p=p_0-\frac{1}{2}$,
meanwhile (\ref{c01}) is transformed into a Hamiltonian system with its symplectic structure $-i\sum\limits_{j\ge 1}dz_j\wedge dz_{-j}$:
\begin{align}\label{c05}
i\dot{z}_j=\frac{\partial H}{\partial\bar{z}_j},\quad -i\dot{\bar{z}}_j=\frac{\partial H}{\partial z_j},\quad \bar{z}_j=z_{-j},
\end{align}
where
\begin{align}\label{c06}\begin{split}
H(z,\bar{z})&=\sum\limits_{j\ge 1}\lambda_j z_j z_{-j}+\frac{1}{120\pi^2}\sum_{\substack{j_1+j_2+j_3+j_4+j_5+j_6=0,\\j_1,j_2,j_3,j_4,j_5,j_6\in\bar{\mathbb{Z}}}}\delta_{j_1}\delta_{j_2}\delta_{j_3}\delta_{j_4}\delta_{j_5} \delta_{j_6} z_{j_1} z_{j_2} z_{j_3} z_{j_4} z_{j_5} z_{j_6}\\
&:=\Lambda+G,
\end{split}
\end{align}
and
\begin{align}\label{c07}
\lambda_j=\frac{j}{1+j^2},\quad |\lambda_j|\approx j^{-1},\quad \text{for } j\in\bar{\mathbb{Z}} .
\end{align}
We have the following results.
\begin{lemma}
For each $j\in\bar{\mathbb{Z}}$, if $\bar{z}_j$ represents the complex conjugate of $z_j$, then the Hamiltonian function $H(z,\bar{z})$ is real valued. The Hamiltonian vector field $X_G$ of the perturbation term $G$ is analytic from $h_p$ to $h_q$ with $q=p+1.$ Furthermore, we have
\begin{align}
\|\lfloor X_G\rceil\|_q\le C\|z\|_p^5.
\end{align}
\end{lemma}
\begin{proof}
Consider $\frac{\partial G}{\partial z_{-j}}$ for $j\in\bar{\mathbb{Z}}$. One has
\begin{align}
\frac{\partial G}{\partial z_{-j}}=6\delta_j\sum\limits_{j_1+j_2+j_3+j_4+j_5=j}\delta_{j_1}\delta_{j_2}\delta_{j_3}\delta_{j_4}\delta_{j_5}z_{j_1}z_{j_2}z_{j_3}z_{j_4}z_{j_5},
\end{align}
then
\begin{align}\begin{split}
|\frac{\partial G}{\partial z_{-j}}|&\le 6\delta_j\sum\limits_{j_1+j_2+j_3+j_4+j_5=j}\delta_{j_1}\delta_{j_2}\delta_{j_3}\delta_{j_4}\delta_{j_5}|z_{j_1}z_{j_2}z_{j_3}z_{j_4}z_{j_5}|\\
&=6\delta_j\cdot(v*v*v*v*v)_j,\\
\end{split}
\end{align}
where $v=(v_j:j\in\bar{\mathbb{Z}})$ with $v_{j_k}=|\delta_{j_k}z_{j_k}| (k=1,2,3,4,5)$ and $v*v$ represents the convolution operation of $v$ and $v$. Recall that $\delta_j\approx j^{-1/2}$, one thus obtains
\begin{align}
\begin{split}
\|\lfloor X_G\rceil\|_q\le \|v*v*v*v*v\|_{q-\frac{1}{2}}\le C\|v\|_{q-\frac{1}{2}}^5=C\|v\|_{p+\frac{1}{2}}^5=C\|z\|_p^5.
\end{split}
\end{align}
The remaining claims are obvious, we omit the details.

\end{proof}

\section{Partial Birkhoff normal form}
We shall use the new KAM theorem in \cite{Yuan2018} to get our desired results, so one has to extract parameters from the six order resonant terms. At the same time, we hope to remove all the six order nonresonant terms to make the perturbation small enough. However, it is difficult to reach. In \cite{Liang2005}, the authors just kill a part of the nonresonant terms with order 6 and obtain a partial Birkhoff normal form. While in \cite{Gao2012}, the authors have to eliminate some nonresonant terms of order 10 so as to avoid all the parameters being excluded. However, here we need to derive a partial Birkhoff normal form of order 14 to apply the new KAM theorem. 
\par
\subsection{Normal form of order six}

Let $S:=\{\pm n_1,\pm n_2|~1\ll n_1\ll n_2\}.$ Split $z=(z_j)_{j\in\mathbb{\bar{Z}}}=(\tilde{z},\hat{z})$ with $\tilde{z}=(z_{n_1},z_{n_2},z_{-n_1},z_{-n_2})$. We define the index sets $\Delta_*$(*=0,1,2) and $\Delta_3$ as follows:
\begin{align*}\begin{split}
\Delta_*&=\{(j_1,j_2,j_3,j_4,j_5,j_6)\in\bar{\mathbb{Z}}^6|\text{ There are just * components of }\\
& \qquad j_1,j_2,j_3,j_4,j_5,j_6 \text{ not in  S}\},\\
\end{split}\end{align*}
\begin{align*}\begin{split}
\Delta_3&=\{(j_1,j_2,j_3,j_4,j_5,j_6)\in\bar{\mathbb{Z}}^6|\text{ There exist at least 3 components of}\\
&\qquad j_1,j_2,j_3,j_4,j_5,j_6\text{ not in S}\}.
\end{split}\end{align*}
We also define the normal formal set in the following form
$$\mathcal{N}:=\{(j_1,j_2,j_3,j_4,j_5,j_6)\in\bar{\mathbb{Z}}^6|(j_1,j_2,j_3,j_4,j_5,j_6)\equiv(a,-a,b,-b,c,-c) \},$$
in which $(j_1,j_2,j_3,j_4,j_5,j_6)=(a,-a,b,-b,c,-c)$ or its possible permutations.
Let us split $G$ into three parts, one has
\begin{align}\label{c07-0}
G=\bar{G}+\tilde{G}+\hat{G},
\end{align}
where $\bar{G}$ represents the normal form part of $G$ fulfilling $(j_1,j_2,j_3,j_4,j_5,j_6)\in(\Delta_0\cup\Delta_1\cup\Delta_2)\cap\mathcal{N}$,
$\tilde{G}$ denotes the non-normal form part of $G$ safisfying $(j_1,j_2,j_3,j_4,j_5,j_6)\in(\Delta_0\cup\Delta_1\cup\Delta_2)\setminus\mathcal{N}$, and $\hat{G}$ represents those terms of $G$ fulfilling $(j_1,j_2,j_3,j_4,j_5,j_6)\in\Delta_3$. Their explicit expressions are listed below respectively,
\begin{align}\label{c07-1}\begin{split}
\bar{G}&=\frac{1}{120\pi^2}\sum\limits_{(j_1,j_2,j_3,j_4,j_5,j_6)\in(\Delta_0\cup\Delta_1\cup\Delta_2)\cap\mathcal{N}}\delta_{j_1}\delta_{j_2}\delta_{j_3}\delta_{j_4}\delta_{j_5}\delta_{j_6} z_{j_1} z_{j_2} z_{j_3} z_{j_4} z_{j_5} z_{j_6}\\
&=\frac{1}{6\pi^2}\biggl[\frac{n_1^3}{(1+n_1^2)^3}|z_{n_1}|^6+\frac{n_2^3}{(1+n_2^2)^3} |z_{n_2}|^6\biggr]\\
\end{split}\end{align}
\begin{align*}\begin{split}
&+\frac{3}{2\pi^2}\biggl[\frac{n_1^2 n_2}{(1+n_1^2)^2(1+n_2^2)} |z_{n_1}|^4|z_{n_2}|^2+\frac{n_1 n_2^2}{(1+n_1^2)(1+n_2^2)^2} |z_{n_1}|^2|z_{n_2}|^4\biggr]\\
&+\frac{3}{2\pi^2}\sum\limits_{j\ge1\text{ and }j\ne n_1,n_2}\frac{j}{1+j^2}\biggl[\frac{n_1^2}{(1+n_1^2)^2}|z_{n_1}|^4|+\frac{n_2^2}{(1+n_2^2)^2}|z_{n_2}|^4\biggr]|z_j|^2\\
&+\frac{6}{\pi^2}\sum\limits_{j\ge1\text{ and }j \ne n_1,n_2}\frac{ n_1 n_2 j}{(1+n_1^2)(1+n_2^2)(1+j^2)} |z_{n_1}|^2|z_{n_2}|^2|z_j|^2.
\end{split}\end{align*}
\begin{align}\label{c07-2}
\tilde{G}=\frac{1}{120\pi^2}\sum\limits_{(j_1,j_2,j_3,j_4,j_5,j_6)\in(\Delta_0\cup\Delta_1\cup\Delta_2)\setminus\mathcal{N}}\delta_{j_1}\delta_{j_2}\delta_{j_3}\delta_{j_4}\delta_{j_5}\delta_{j_6} z_{j_1} z_{j_2} z_{j_3} z_{j_4} z_{j_5} z_{j_6},
\end{align}
\begin{align}\label{c07-3}
\hat{G}=\frac{1}{120\pi^2}\sum\limits_{(j_1,j_2,j_3,j_4,j_5,j_6)\in\Delta_3}\delta_{j_1}\delta_{j_2}\delta_{j_3}\delta_{j_4}\delta_{j_5}\delta_{j_6} z_{j_1} z_{j_2} z_{j_3} z_{j_4} z_{j_5} z_{j_6}.
\end{align}

To remove the part $\tilde{G}$, We need the following proposition.
\begin{proposition}\label{p3.1}
For each $(j_1,j_2,j_3,j_4,j_5,j_6)\in(\Delta_0\cup\Delta_1\cup\Delta_2)\setminus\mathcal{N}$, one has
\begin{align}\label{c07-4}
|\lambda_{j_1}+\lambda_{j_2}+\lambda_{j_3}+\lambda_{j_4}+\lambda_{j_5}+\lambda_{j_6}|\ge C(n_1,n_2),
\end{align}
where $C(n_1,n_2)>0$ is a constant depending only on $n_1$ and $n_2$.
\end{proposition}

\begin{proof}
We shall discuss the proposition into several cases.
\begin{description}
\item Case(I). When $j_5+j_6=0$.
\end{description}
At this time one has $$|\lambda_{j_1}+\lambda_{j_2}+\lambda_{j_3}+\lambda_{j_4}+\lambda_{j_5}+\lambda_{j_6}|=|\lambda_{j_1}+\lambda_{j_2}+\lambda_{j_3}+\lambda_{j_4}|.$$
\begin{description}
\item Case(I-1). $j_5, j_6\in\{\pm n_1,\pm n_2\}.$
\end{description}
\begin{itemize}
\item Three components of $j_1,j_2,j_3,j_4$ are negative, the other one is positive. Without loss of generality, we assume $j_1>0$ and $j_2,j_3,j_4<0$.\par
From the fact that $j_1+j_2+j_3+j_4=0$ one can deduce $j_1=|j_2|+|j_3|+|j_4|$. Since the function $f(t)=\frac{t}{1+t^2}$ is monotone decreasing, then
\begin{align}\label{c08}\begin{split}
|\lambda_{j_1}+\lambda_{j_2}+\lambda_{j_3}+\lambda_{j_4}|&=\frac{|j_2|}{1+j_2^2}+\frac{|j_3|}{1+j_3^2}+\frac{|j_4|}{1+j_4^2}-\frac{j_1}{1+j_1^2}\\
&\ge\frac{2n_2}{1+n_2^2}.
\end{split}\end{align}

\item Two components of $j_1,j_2,j_3,j_4$ are negative, while the remaining ones are positive. Without loss of generality, one assumes that $j_1,j_4>0,j_2,j_3<0$ and $j_1\le |j_2|\le|j_3|\le j_4.$\par
It is easily check that the function $f(t)=\frac{t}{1+t^2}$ is convex for $t\ge 2.$ Then we shall discuss it into two subcases. On one hand, when $j_1\geq2,$ it follows that
\begin{align}\label{c10}\begin{split}
&|\lambda_{j_1}+\lambda_{j_2}+\lambda_{j_3}+\lambda_{j_4}|\\
=&|f(j_1)+f(j_4)-f(|j_2|)-f(|j_3|)|\\
=&|f(j_1)-f(|j_2|)-(f(|j_3|)-f(j_4))|\\
\ge&| f(j_1)-f(|j_2|)-(f(|j_2|)-f(j_4+|j_3|-|j_2|))|\\
=&|f(j_1)-2f(|j_2|)+f(2|j_2|-j_1)|\\
\ge& |f''(2n_2)|:=C(n_1,n_2),
\end{split}\end{align}
the last inequality follows from the fact that $j_1\le n_2.$

On the other hand, if $j_1=1$, one has $1+j_4=|j_2|+|j_3|.$ Notice that at least two components of $j_2,j_3,j_4$ lie in $S,$ it deduces that $|j_2|,|j_3|,j_4\geq n_1$. Thus,
\begin{equation}\label{c09}
  |\lambda_{j_1}+\lambda_{j_2}+\lambda_{j_3}+\lambda_{j_4}|=\frac{1}{2}+O(\frac{1}{n_1})\geq\frac{1}{4}.
\end{equation}
\end{itemize}
\begin{description}
\item Case(I-2). $j_5,j_6\notin\{\pm n_1,\pm n_2\}.$
\end{description}
This case is trivial.
\begin{description}
\item Case (II). When $j_5+j_6\ne 0.$
	\end{description}
\begin{itemize}
\item All the components of $j_1,j_2,j_3,j_4,j_5,j_6$ belong to $S$.\par
At this time, $j_1+j_2+j_3+j_4+j_5+j_6=rn_1-(6-r)n_2=0,$ where $1\le r\le 5.$ Thus
$$|\lambda_{j_1}+\lambda_{j_2}+\lambda_{j_3}+\lambda_{j_4}+\lambda_{j_5}+\lambda_{j_6}|=\frac{rn_1}{1+n_1^2}-\frac{(6-r)n_2}{1+n^2_2}>C(n_1,n_2)$$
holds true taking account of $n_1\ll n_2.$
\item Only one component of $j_1,j_2,j_3,j_4,j_5,j_6$ lies outside of $S$. \par
Without loss of generality, suppose that the component $s$ lies outside of $S,$ then one finds that
$rn_1\pm(5-r)n_2+j_6=0$ with $0\le r\le 5.$ If $r=0$ or $r=5$, we conclude it obviously, otherwise we have $j_6>\frac{n_2}{2}$. Therefore,
\begin{align}\begin{split}\label{c11}
&|\lambda_{j_1}+\lambda_{j_2}+\lambda_{j_3}+\lambda_{j_4}+\lambda_{j_5}+\lambda_{j_6}|\\
=&\big|\frac{rn_1}{1+n_1^2}-\frac{(5-r)n_2}{1+n_2^2}+\frac{j_6}{1+j_6^2}\big|\\
>&\frac{rn_1}{2(1+n_1^2)},
\end{split}\end{align}
where we have used the fact that $n_1\ll n_2$ again.
\item Two components of $j_1,j_2,j_3,j_4,j_5,j_6$ lie outside of $S$.\par
Without loss of generality, suppose that $j_5,j_6\not\in S$, $j_1,j_2,j_3,j_4\in S$ and $|j_5|\leq |j_6|.$ Then one has
$$rn_1\pm(4-r)n_2+j_5+j_6=0$$ with $0\le r\le 4,$ denote $b=|j_5|, d=|j_6|.$
\begin{description}
\item Subcase(II-1). $rn_1+(4-r)n_2+j_5+j_6=0$ with $0\leq r\leq4.$\\
It is clear that $j_5+j_6<0.$ Since $|j_5|\leq |j_6|,$ it happens $j_5>0, j_6<0$ or $j_5<0, j_6<0$. For the former case, we adopt the similar method as that in the first subcase of Case(I-1). While for the latter case we have $rn_1+(4-r)n_2=b+d.$ It is divided into two subcases to discuss. \\
Subcase(II-1-1). $rn_1+(4-r)n_2=b+d$ with $r=0\text{ or }4.$\\
We only consider the case with $r=4$, the case with $r=0$ can be handled similarly.
In this case, one has $4n_1=d+b,$ and
\begin{align}\begin{split}
&\big|\lambda_{j_1}+\lambda_{j_2}+\lambda_{j_3}+\lambda_{j_4}+\lambda_{j_5}+\lambda_{j_6}\big|\\
=&\big|\frac{4n_1}{1+n_1^2}-\frac{d}{1+d^2}-\frac{b}{1+b^2}\big|\\
=&4n_1\big|\frac{b^2d^2+(b+d)^2-(3+n_1^2)bd-n_1^2}{(1+n_1^2)(1+d^2)(1+b^2)}\big|\\
=&4n_1\frac{|(bd)^2-(3+n_1^2)bd+15n_1^2|}{(1+n_1^2)(1+d^2)(1+b^2)}.
\end{split}\end{align}
It remains to show that $(bd)^2-(3+n_1^2)bd+15n_1^2\ne 0,$ which turns to prove that the equation $x^2-(3+n_1^2)x+15n_1^2=0$ does not admit
integer solutions. Actually, when $n_1\ge 20,$ the discriminant of equation $x^2-(3+n_1^2)x+15n_1^2=0$ takes the form $\Delta=(n_1^2-27)^2-720$,
which cannot be a square of some integer. Otherwise, one has
$$\Gamma\le (n_1^2-27-1)^2,$$ which infers $n_1^2\le 366+27$ and $n_1\le 19,$ this leads to contradiction. Note that $b, d<4n_1$, therefore, it implies
$$\big|\lambda_{j_1}+\lambda_{j_2}+\lambda_{j_3}+\lambda_{j_4}+\lambda_{j_5}+\lambda_{j_6}\big|\ge\frac{4n_1}{(1+n_1^2)(1+16n_1^2)^2}.$$
Subcase(II-1-2). $rn_1+(4-r)n_2=b+d$ with $ 1\leq r\leq 3$.

When $b\geq rn_1,$ one has $rn_1\le b\le d\le(4-r)n_2,$ hence
\begin{align*}\begin{split}
&|\lambda_{j_1}+\lambda_{j_2}+\lambda_{j_3}+\lambda_{j_4}+\lambda_{j_5}+\lambda_{j_6}|\\
=&\big|\frac{rn_1}{1+n_1^2}+\frac{(4-r)n_2}{1+n_2^2}-\frac{d}{1+d^2}-\frac{b}{1+b^2}\big|\\
\ge& \frac{rn_1}{1+n_1^2}+\frac{(4-r)n_2}{1+n_2^2}-\frac{d}{1+d^2}-\frac{b}{1+b^2}\\
= &\frac{rn_1}{1+n_1^2}-f(rn_1)+\frac{(4-r)n_2}{1+n_2^2}-f((4-r)n_2)+f(rn_1)\\
&\qquad+f((4-r)n_2)-f(b)-f(d)\\
=&\frac{rn_1^3(r^2-1)}{(1+n_1^2)(1+r^2 n_1^2)}+\frac{(4-r)n_2^3[(4-r)^2-1]}{(1+n_2^2)[1+(4-r)^2 n_2^2]}+f''(\theta)(rn_1-b)^2\\
>&\frac{rn_1^3(r^2-1)}{(1+n_1^2)(1+r^2 n_1^2)}+\frac{(4-r)n_2^3[(4-r)^2-1]}{(1+n_2^2)[1+(4-r)^2 n_2^2]}\\
:=&C(n_1,n_2)>0.
\end{split}\end{align*}
When $b<rn_1$, then one obtains $b<rn_1<(4-r)n_2<d.$ It follows that
\begin{align*}\begin{split}
&\big|\lambda_{j_1}+\lambda_{j_2}+\lambda_{j_3}+\lambda_{j_4}+\lambda_{j_5}+\lambda_{j_6}\big|\\
=&\big|\frac{b}{1+b^2}-\frac{rn_1}{1+n_1^2}-[\frac{(4-r)n_2}{1+n_2^2}-\frac{d}{1+d^2}]\big|\\
\geq&\big|\frac{rn_1b^2-(1+n_1^2)b+rn_1}{(1+b^2)(1+n_1^2)}\big|-\big|\frac{(4-r)n_2}{1+n_2^2}-\frac{d}{1+d^2}\big|\\
:=&I_1-I_2.
\end{split}\end{align*}
We claim that $rn_1b^2-(1+n_1^2)b+rn_1\ne 0.$ In fact, if $r=1$, then $b<n_1$ and
\begin{equation*}
  rn_1b^2-(1+n_1^2)b+rn_1=(b-n_1)(n_1b-1)<0.
\end{equation*}
If $2\le r\le 3,$ we consider the equation $rn_1x^2-(1+n_1^2)x+rn_1=0.$ The corresponding discriminant $\Delta=(n_1^2-(2r^2-1))^2-(2r^2-1)^2+1$ cannot be a square of some integer as $n_1\gg1$. Thus, by the fact $b<rn_1$, one gets
\begin{equation*}
  I_1\geq\frac{1}{(1+r^2n_1^2)^2}.
\end{equation*}
Since $d>(4-r)n_2$, it is easily know that $I_2=O(\frac{1}{n_2})$. Taking account of $n_1\ll n_2,$ we conclude that
$$|\lambda_{j_1}+\lambda_{j_2}+\lambda_{j_3}+\lambda_{j_4}+\lambda_{j_5}+\lambda_{j_6}|\ge \frac{1}{2(1+r^2n_1^2)^2}.$$


\item Subcase(II-2). $rn_1-(4-r)n_2+j_5+j_6=0$ with $0\leq r\leq4.$\par
Subcase(II-2-1). $rn_1-(4-r)n_2+j_5+j_6=0$ with $1\leq r\leq3$.\\
In this case, $s$ must be positive. If $j_5>0$, then one obtains
\begin{align*}\begin{split}
&|\lambda_{j_1}+\lambda_{j_2}+\lambda_{j_3}+\lambda_{j_4}+\lambda_{j_5}+\lambda_{j_6}|\\
=&|\frac{rn_1}{1+n_1^2}+\frac{b}{1+b^2}+\frac{d}{1+d^2}-\frac{(4-r)n_2}{1+n_2^2}|\\
\ge& \frac{rn_1}{2(1+n_1^2)}:=C(n_1,n_2)>0.
\end{split}\end{align*}
While, for $j_5<0,$ we have $rn_1+d=(4-r)n_2+b.$

Similarly, we firstly discuss the case $b>rn_1.$ It is easy to know $rn_1<b\le(4-r)n_2<d.$ Then
\begin{align*}\begin{split}
&|\lambda_{j_1}+\lambda_{j_2}+\lambda_{j_3}+\lambda_{j_4}+\lambda_{j_5}+\lambda_{j_6}|\\
=&|\frac{rn_1}{1+n_1^2}+\frac{d}{1+d^2}-\frac{b}{1+b^2}-\frac{(4-r)n_2}{1+n_2^2}|\\
\geq&rf(n_1)-f(rn_1+1)-(4-r)f(n_2)+f((4-r)n_2+1)\\
&\qquad +f(rn_1+1)-f(b)-f((4-r)n_2+1)+f(d)\\
=&rf(n_1)-f(rn_1+1)-(4-r)f(n_2)+f((4-r)n_2+1)\\
&\qquad+f''(\theta)(rn_1+1-b)^2,
\end{split}\end{align*}
where $\theta\in (rn_1+1,d)$. Due to the monotonicity of the function $f(x)$, we obtain that
\begin{equation*}
  rf(n_1)-f(rn_1+1)\geq f(n_1)-f(n_1+1)\geq\frac{1}{1+(n_1+1)^2}.
\end{equation*}
Since $f((4-r)n_2+1)-(4-r)f(n_2)=O(\frac{1}{n_2})$ and $f''(\theta)\ge0,$
hence, $$|\lambda_{j_1}+\lambda_{j_2}+\lambda_{j_3}+\lambda_{j_4}+\lambda_{j_5}+\lambda_{j_6}|\ge \frac{1}{2+2(n_1+1)^2}.$$
When $b\le rn_1,$ one has $b\le rn_1\le d\le(4-r)n_2.$ And it follows that
\begin{align*}\begin{split}
&\big|\lambda_{j_1}+\lambda_{j_2}+\lambda_{j_3}+\lambda_{j_4}+\lambda_{j_5}+\lambda_{j_6}\big|\\
=&\big|\frac{rn_1(1+b^2)-b(1+n_1^2)}{(1+n_1^2)(1+b^2)}+\frac{d(1+n_2^2)-(4-r)n_2(1+d^2)}{(1+d^2)(1+n_2^2)}\big|\\
:=&|I_3+I_4|.
\end{split}\end{align*}
If $b<rn_1$ or $2\le r\le 3$, by the same method as that in the Subcase(II-1-1) one can show  $rn_1(1+b^2)-b(1+n_1^2)\ne 0.$ In view of $d=(4-r)n_2+b-rn_1$ and $n_1\ll n_2,$ we know that $d>\frac{n_2}{2}$ and $I_4=O(\frac{1}{n_2})$. So it deduces that
 $$|\lambda_{j_1}+\lambda_{j_2}+\lambda_{j_3}+\lambda_{j_4}+\lambda_{j_5}+\lambda_{j_6}|\ge|I_3|-|I_4|>\frac{1}{2(1+r^2n_1^2)^2}.$$
If $b=rn_1$ and $r=1$, then $d=3n_2$. Now it is clear that  $$|\lambda_{j_1}+\lambda_{j_2}+\lambda_{j_3}+\lambda_{j_4}+\lambda_{j_5}+\lambda_{j_6}|=|I_4|=\frac{24n_2^3}{(1+n_2^2)(1+9n_2^2).}$$
Subcase (II-2-2). $rn_1-(4-r)n_2+b+d=0$ with $r=0,r=4.$\par
It is similar with the case (II-1-2).
	\end{description}
This completes the proof.
\end{itemize}
\end{proof}

Based on the proposition \ref{p3.1}, we obtain the following partial Birkhoff normal form proposition.
\begin{proposition}\label{thm1}
There exists a real analytic symplectic coordinate transformation $\Phi_1$ which maps the neighborhood of the origin of $h_p$ to $h_q$, such that
the Hamiltonian H is transformed into a partial Birkhoff normal form up to order six. More precisely, one has
\begin{align}\label{c12}
H\circ\Phi_1=\Lambda+\bar{G}+\hat{G}+R,
\end{align}
where
\begin{align}\label{c15}
 |\hat{G}|=O(\|z\|_p^{3}\|\hat{z}\|_p^{3})\text{ and }|R|=O(\|z\|_p^{10}).
\end{align}
\end{proposition}

\begin{proof}
Let $\Phi_1=X_F^1$ be the time-1 map of the Hamiltonian flow of vector field $X_F$ defined by the Hamiltonian
$$F=\sum\limits_{(j_1,j_2,j_3,j_4,j_5,j_6)\in\bar{\mathbb{Z}}^6} F_{j_1j_2j_3j_4j_5j_6}z_{j_1} z_{j_2} z_{j_3} z_{j_4} z_{j_5} z_{j_6}$$
with the coefficients
\begin{align*}
iF_{jklmns}=\left\{\begin{array}{cc}
\frac{1}{120\pi^2}\frac{\delta_{j_1}\delta_{j_2}\delta_{j_3}\delta_{j_4}\delta_{j_5}\delta_{j_6}}{\lambda_{j_1}+\lambda_{j_2}+\lambda_{j_3}+\lambda_{j_4}+\lambda_{j_5}+\lambda_{j_6}},(j_1,j_2,j_3,j_4,j_5,j_6)\in(\Delta_0\cup\Delta_1\cup\Delta_2)\setminus\mathcal{N},\\
0,\qquad otherwise.\\
\end{array}\right.
\end{align*}
It follows from the Proposition 1.1 that $\lambda_{j_1}+\lambda_{j_2}+\lambda_{j_3}+\lambda_{j_4}+\lambda_{j_5}+\lambda_{j_6}\ne 0$ when $j_1+j_2+j_3+j_4+j_5+j_6=0.$ Thus $F$ is well defined.

We firstly establish the regularity of Hamiltonian vector field $X_F$.
By the Proposition 1.1, we have
\begin{align}\label{c16}
|iF_{j_1j_2j_3j_4j_5j_6}|\le \frac{\delta_{j_1} \delta_{j_2} \delta_{j_3} \delta_{j_4} \delta_{j_5} \delta_{j_6}}{120\pi^2 C(n_1,n_2)}.
\end{align}
The $j$-th entry of the vector field $X_F$ takes the following form
\begin{align}\begin{split}
 i\sigma_j\frac{\partial F}{\partial z_{-j}}&=i\sigma_j\sum\limits_{j_2+j_3+j_4+j_5+j_6=j}[F_{(-j)j_2j_3j_4j_5j_6}+F_{j_2(-j)j_3j_4j_5j_6}\\
 &+F_{j_2j_3(-j)j_4j_5j_6}+F_{j_2j_3j_4(-j)j_5j_6}+F_{j_2j_3j_4j_5(-j)j_6}+F_{j_2j_2j_4j_5j_6(-j)}],\\
 \end{split}
\end{align}
therefore one has
\begin{align*}\begin{split}
|i\sigma_j\frac{\partial F}{\partial z_{-j}}|&\le\frac{1}{20\pi^2 C(n_1,n_2)}\delta_j\sum\limits_{j_2+j_3+j_4+j_5+j_6=j}\delta_{j_2}\delta_{j_3}\delta_{j_4}\delta_{j_5}\delta_{j_6}|z_{j_2} z_{j_3} z_{j_4} z_{j_5} z_{j_6}|\\
&\le \frac{1}{20\pi^2 C(n_1,n_2)}\frac{1}{\sqrt{|j|}}\sum\limits_{j_2+j_3+j_4+j_5+j_6=j}\frac{|z_{j_2}|}{\sqrt{|j_2|}}\frac{|z_{j_3}|}{\sqrt{|j_3|}}\frac{|z_{j_4}|}{\sqrt{|j_4|}}\frac{|z_{j_5}|}{\sqrt{|j_5|}}\frac{|z_{j_6}|}{\sqrt{|j_6|}}.
\end{split}\end{align*}
Let $w=(\frac{|z_{j_2}|}{\sqrt{|j_2|}}:j_2\in\bar{\mathbb{Z}}),$ hence
\begin{align}\label{c17}\begin{split}
\|\lfloor X_{F}\rceil\|_q&=\|\lfloor X_{F}\rceil\|_{p+1}\\
&\le \frac{1}{20\pi^2 C(n_1,n_2)}\|w*w*w*w*w\|_{p+\frac{1}{2}}\\
&\le \frac{1}{20\pi^2 C(n_1,n_2)}\|w\|_{p+\frac{1}{2}}^5\\
&=\frac{1}{20\pi^2 C(n_1,n_2)}\|z\|_{p}^5,
\end{split}\end{align}
where $*$ represents the convolution operation. It follows that $X_{F}$ is a real analytic vector field mapping a small neighborhood of the origin in $h_p$ to $h_q$, meanwhile $\Phi_1=X_F^1$ is a symplectic coordinate transformation defined in the neighborhood of the origin in $h_p$.

Secondly, using the Taylor's formula, one has
\begin{align}\label{c18}
\begin{split}
H\circ\Phi_1&=H\circ X_F^t|_{t=1}\\
&=H+\{\Lambda,F\}+\int_0^1 (1-t)\{\{\Lambda,F\},F\}\circ X_F^t dt+\int_0^1\{G,F\}\circ X_F^t dt\\
&=\Lambda +\tilde{G}+\{\Lambda,F\}+\bar{G}+\hat{G}\\
&\qquad+\int_0^1 (1-t)\{\{\Lambda,F\},F\}\circ X_F^t dt+\int_0^1\{G,F\}\circ X_F^t dt.
\end{split}\end{align}
By direct computation, we have
\begin{align}\label{c20}
\tilde{G}+\{\Lambda,F\}=0.
\end{align}
Taking account of the relation (\ref{c20}), one obtains
\begin{align}\label{c21}
H\circ\Phi_1=\Lambda+\bar{G}+\hat{G}+\int_0^1\{\bar{G}+\hat{G}+t\tilde{G},F\}\circ X_F^t dt.
\end{align}
Denote
\begin{align}\label{c22}
R:=\int_0^1\{\bar{G}+\hat{G}+t\tilde{G},F\}\circ X_F^t dt.
\end{align}
It is easy to verify that
\begin{align}
 |\hat{G}|=O(\|z\|_p^{3}\|\hat{z}\|_p^{3})\text{ and }|R|=O(\|z\|_p^{10}),
\end{align}
which completes the proof.
\end{proof}

\subsection{Normal form of order ten}
To obtain the partial Birkhoff normal form of order ten, one requires to define some new index sets in the following way:
\begin{align*}\begin{split}
\mathcal{N}'&=\{(j_1,\cdots,j_{10})\in\bar{\mathbb{Z}}^{10}: \text{There exists a 10-permutation $\tau$ such that} \\
&\quad j_{\tau(1)}=-j_{\tau(2)},j_{\tau(3)}=-j_{\tau(4)},j_{\tau(5)}=-j_{\tau(6)},j_{\tau(7)}=-j_{\tau(8)},j_{\tau(9)}=-j_{\tau(10)}\},
\end{split}\end{align*}
\begin{align*}\begin{split}
\Delta_0'&=\{(j_1,\cdots,j_{10})\in\bar{\mathbb{Z}}^{10}:\text{All the components of $j_1,\cdots,j_{10}$ are in}\\
&\qquad\{\pm n_1, \pm n_2\}\},
\end{split}\end{align*}
\begin{align*}\begin{split}
\Delta_1'&=\{(j_1,\cdots,j_{10})\in\bar{\mathbb{Z}}^{10}:\text{There is only one component of $j_1,\cdots,j_{10}$} \\
&\quad\text{not in }\{\pm n_1, \pm n_2\}\},\\
\Delta_2'&=\{(j_1,\cdots,j_{10})\in\bar{\mathbb{Z}}^{10}:\text{There are at least two components of $j_1,\cdots,j_{10}$ }\\
&\qquad\text{not in }\{\pm n_1, \pm n_2\}\}.
\end{split}\end{align*}
By what we have just defined, one has the following lemma.
\begin{lemma}\label{c22-00}
Assume that $n_1, n_2\in S$, then for $(j_1,j_2,\cdots,j_{10})\in(\Delta_0'\cup\Delta_1')\setminus\mathcal{N}'$ satisfying $j_1+j_2+\cdots+j_{10}=0$, we have
\begin{align*}
|\lambda_{j_1}+\lambda_{j_2}+\cdots\lambda_{j_{10}}|\ge C(n_1,n_2),
\end{align*}
in which $C(n_1,n_2)$ denotes a positive constant depending only on $n_1, n_2.$
\end{lemma}
\begin{proof}
In view of the facts that $n_1\ll n_2$ and $j_1+j_2+\cdots+j_{10}=0$, it follows easily that the set $\Delta_0'\setminus\mathcal{N}'$ is empty, hence it suffices to consider the case when $(j_1,j_2,\cdots,j_{10})\in\Delta_1'$. Without loss of generality, we assume that the component $j_{10}\notin\{\pm n_1,\pm n_2\}$. Let $\alpha,\bar{\alpha},\beta,\bar{\beta}$ be the number of $n_1,-n_1,n_2,-n_2$ in $\{j_1,j_2,\cdots,j_9\}$ respectively. Thus one has
\begin{align}\label{c22-0}
(\alpha-\bar{\alpha})n_1+(\beta-\bar{\beta})n_2+j_{10}=0,
\end{align}
and
\begin{align}
\alpha+\bar{\alpha}+\beta+\bar{\beta}=9.
\end{align}

\begin{itemize}
\item Case 1: $\alpha\neq\bar{\alpha}$ and $\beta\neq\bar{\beta}.$ \\
From (\ref{c22-0}), we obtain that $|j_{10}|\ge\frac{|\beta-\bar{\beta}|}{2} n_2.$ Therefore, one has
\begin{align*}
\begin{split}
|\lambda_{j_1}+\lambda_{j_2}+\cdots+\lambda_{j_{10}}|&\ge |\alpha-\bar{\alpha}|\frac{n_1}{1+n_1^2}-|\beta-\bar{\beta}|\frac{n_2}{1+n_2^2}-\frac{|j_{10}|}{1+j_{10}^2}\\
&\ge \frac{|\alpha-\bar{\alpha}|}{2}\frac{n_1}{1+n_1^2}.
\end{split}\end{align*}
\item Case 2: $\alpha\neq\bar{\alpha}$ and $\beta=\bar{\beta}$.\\
Now one has $(\alpha-\bar{\alpha})n_1+j_{10}=0,$ which implies that $|\alpha-\bar{\alpha}|\ge 2.$
In fact, if $|\alpha-\bar{\alpha}|=1,$ one has $j_{10}=\pm n_1,$ which leads to contradiction.
Thus we have
\begin{align}
|\lambda_{j_1}+\cdots+\lambda_{j_{10}}|\ge \frac{2n_1}{1+n_1^2}-\frac{2n_1}{1+4n_1^2}\ge C(n_1).
\end{align}
\item Case 3: $\alpha=\bar{\alpha}$ and $\beta\neq\bar{\beta}$.\\
Observe that $(\beta-\bar{\beta})n_2+j_{10}=0,$ by an analogous discussion, one has
\begin{align}
|\lambda_{j_1}+\cdots+\lambda_{j_{10}}|\ge \frac{2n_2}{1+n_2^2}-\frac{2n_2}{1+4n_2^2}\ge C(n_2).
\end{align}
\item Case 4: $\alpha=\bar{\alpha}$ and $\beta=\bar{\beta}$.\\
This case will never happen in view of $\alpha+\bar{\alpha}+\beta+\bar{\beta}=9.$
\end{itemize}

\end{proof}

By simple calculation, the formula of (\ref{c22}) can be rewritten as
\begin{align}\label{c22-2}
R=\{\bar{G}+\hat{G}+\frac{1}{2}\tilde{G},F\}+\int_0^1\{\{(1-t)(\hat{G}+\bar{G}+\frac{1}{2}(1-t^2)\tilde{G},F\},F\}\circ X_F^t dt,
\end{align}
in which the order of the first term is 10, and the order of the second term is at least 14. To obtain the partial Birkhoff normal form of order 10, it suffices to introduce another real analytic symplectic transformation. To do this, let us split the first term of $R$ in (\ref{c22-2}) into the following three parts:
\begin{align}\label{c22-3}
\{\bar{G}+\hat{G}+\frac{1}{2}\tilde{G},F\}=\bar{R}+\tilde{R}+\hat{R},
\end{align}
in which $\bar{R}$ denotes the normal form part with $(j_1,\cdots,j_{10})\in\Delta_0'\cap\mathcal{N}'$, while $\tilde{R}$ is the non-normal form term with $(j_1,\cdots,j_{10})\in(\Delta_0'\cup\Delta_1')\setminus\mathcal{N}'$, and meanwhile $\hat{R}$ represents the part with $(j_1,\cdots,j_{10})\in\Delta_2'.$

By lemma \ref{c22-00}, one can find a symplectic transformation to eliminate the term of order 10, that is, $\tilde{R}$. Thus one can further obtain a partial Birkhoff normal form of order 10, which is stated as follows:
\begin{proposition}\label{thm2}
Assume $n_1\ll n_2$, then by a symplectic transformation $\Phi_2$, which is real analytic in some neiborhood of the origin from $h_p$ to $h_q,$ the Hamiltonian function
$H\circ\Phi_1$ in (\ref{c12}) is changed into
\begin{align}\label{y1}
H\circ\Phi_1\circ\Phi_2=\Lambda+\bar{G}+\hat{G}+\bar{R}+\hat{R}+T,
\end{align}
where $\bar{R}$ has the following form
\begin{align}\label{y2}
\begin{split}
\bar{R}&=R_0|z_{n_1}|^{10}+R_1|z_{n_1}|^8|z_{n_2}|^2+R_2|z_{n_1}|^6|z_{n_2}|^4+R_3|z_{n_1}|^4|z_{n_2}|^6\\
&+R_4|z_{n_1}|^2|z_{n_2}|^8+R_5|z_{n_2}|^{10},
\end{split}
\end{align}
with coefficients $R_0,\cdots,R_5$ real and depending only on $n_1$ and $n_2.$
Moreover,
\begin{align}\label{y3}
|\hat{R}|=O(\|z\|_p^8\|\hat{z}\|_p^2),
\end{align}
\begin{align}\label{y4}
|T|=O(\|z\|_p^{14}).
\end{align}
The Hamiltonian vector fields $X_{\bar{R}},X_{\hat{R}},X_T$ are analytic from $h_p$ to $h_q.$
\end{proposition}
\begin{proof}
The proof of this Proposition is similar with Proposition \ref{thm1}, we omit it.
\end{proof}
\subsection{Normal form of order 14}
In this subsection, we shall remove some terms of order 14 to guarantee that the perturbation satisfies the \textbf{Assumption C} for the KAM theorem \ref{thmA}. For this, we define the normal form set
\begin{align*}\begin{split}
\mathcal{N}''&=\{(j_1,\cdots,j_{14})\in\bar{\mathbb{Z}}^{14}: \text{ There exists a 14-permutation $\tau$ such that } \\
& j_{\tau(1)}=-j_{\tau(2)},j_{\tau(3)}=-j_{\tau(4)},j_{\tau(5)}=-j_{\tau(6)},j_{\tau(7)}=-j_{\tau(8)},j_{\tau(9)}=-j_{\tau(10)},\\
&j_{\tau(11)}=-j_{\tau(12)},j_{\tau(13)}=-j_{\tau(14)}\},
\end{split}\end{align*}
and the index sets as follows:
\begin{align*}
\Delta_0''=\{(j_1,\cdots,j_{14})\in\bar{\mathbb{Z}}^{14}:\text{All the components are in } \{\pm n_1,\pm n_2\}\},
\end{align*}
\begin{align*}
\Delta_1''=\{(j_1,\cdots,j_{14})\in\bar{\mathbb{Z}}^{14}:\exists\text{ at least one component not in } \{\pm n_1,\pm n_2\}\}.
\end{align*}

By direct computation, the expression of $T$ in proposition \ref{thm2} reads
\begin{align}\label{y6}\begin{split}
T&=\{\{\frac{1}{2}(\bar{G}+\hat{G})+\frac{1}{3}\tilde{G},F\},F\}\\
&+\int_0^1\{\{\{\frac{1}{2}(t-1)^2(\bar{G}+\hat{G})+\frac{1}{6}(t^3-3t+2)\tilde{G},F\},F\},F\}\circ X_F^t dt,\\
\end{split}
\end{align}
where the order of the first term is 14, and the order of the second term is at least 18.

Let us split the first term of $T$ in (\ref{y6}) into three parts:
\begin{align}\label{y7}
\{\{\frac{1}{2}(\bar{G}+\hat{G})+\frac{1}{3}\tilde{G},F\},F\}=\bar{T}+\tilde{T}+\hat{T},
\end{align}
in which $\bar{T}$ is the normal form part with $(j_1,\cdots,j_{14})\in\Delta_0''\cap\mathcal{N}'',$ $\tilde{T}$ is the non-normal form term fulfilling $(j_1,\cdots,j_{14})\in\Delta_0''\cup\mathcal{N}''$, and $\hat{T}$ denotes the part with $(j_1,\cdots,j_{14})\in\Delta_1''.$ Next, it suffices to seek a symplectic transformation to eliminate the term $\tilde{T}$. To this end, it requires to establish a lemma about the divisor $\lambda_{j_1}+\cdots\lambda_{j_{14}}$:
\begin{lemma}\label{y8}
Suppose $n_1\ll n_2,$ then for $(j_1,\cdots,j_{14})\in\Delta_0''\setminus\mathcal{N}''$ fulfilling $j_1+\cdots+j_{14}=0,$ one has
\begin{align}\label{y9-0}
|\lambda_{j_1}+\lambda_{j_2}+\cdots+\lambda_{j_{14}}|\ge C(n_1,n_2).
\end{align}
\end{lemma}
By an argument similar to that of Lemma \ref{c22-00}, we can prove this lemma, we omit the details.

By applying the Lemma \ref{y8}, one can further arrive at the following results.
\begin{proposition}\label{thm3}
Assume $n_1\ll n_2.$ Then by another symplectic transformation $\Phi_3$, which is real analytic in some neighborhood of the origin from $h_p$ to $h_q,$ the Hamitonian function $H\circ\Phi_1\circ\Phi_2$ in (\ref{y1}) transformed into a partial Birkhoff normal form of order 14 in the following form
\begin{align}\label{y9}
H\circ\Phi_1\circ\Phi_2\circ\Phi_3=\Lambda+\bar{G}+\hat{G}+\bar{R}+\hat{R}+\bar{T}+\hat{T}+W,
\end{align}
in which $\bar{T}$ reads
\begin{align}\label{y10}\begin{split}
\bar{T}&=T_0|z_{n_1}|^{14}+T_1|z_{n_1}|^{12}|z_{n_2}|^2+T_2|z_{n_1}|^{10}|z_{n_2}|^4+T_3|z_{n_1}|^{8}|z_{n_2}|^6  \\
&+T_4|z_{n_1}|^{6}|z_{n_2}|^8+T_5|z_{n_1}|^{4}|z_{n_2}|^{10}+T_6|z_{n_1}|^{2}|z_{n_2}|^{12}+T_7||z_{n_2}|^{14}
\end{split}\end{align}
with real coefficents $T_0,\cdots,T_7$ depending only on $n_1$ and $n_2,$
and $W$ is of order at least 18 with the following form
\begin{align}\label{y11}
W=\int_0^1\{\{\{\frac{1}{2}(t-1)^2(\bar{G}+\hat{G})+\frac{1}{6}(t^3-3t+2)\tilde{G},F\},F\},F\}\circ X_F^t dt,
\end{align}
\begin{align}\label{y12}
|\hat{T}|=O(\|z\|_p^{13}\|\hat{z}\|_p),
\end{align}
\begin{align}\label{y13}
|W|=O(\|z\|_p^{18}).
\end{align}
Furthermore, the Hamiltonian vector fields $X_{\bar{T}},X_{\hat{T}},X_{W}$ are analytic from $h_p$ to $h_q.$
\end{proposition}
\begin{proof}
This Proposition can be proved similarly as Propositon \ref{thm1}, we omit it.
\end{proof}

\section{Proof of the main Theorem}
In this section, we shall give the proof of our main Theorem \ref{t1} by applying a new KAM Theorem \ref{thmA} developed by the third author in \cite{Yuan2018}, which is stated in the Appendix.

To this end, let us introduce the action-angle varialbles $(y,x)$ as follows
\begin{align}\label{y14}
\left\{\begin{array}{cc}
z_{n_j}=\sqrt{\xi_j^{1/2}+y_j}e^{-ix_j},&\quad z_{-n_j}=\sqrt{\xi_j^{1/2}+y_j}e^{ix_j},\quad j=1,2,\\
z_j=z_j,\quad\qquad &j\ne \pm n_1,\pm n_2,\\
\end{array}\right.
\end{align}
where $\xi=(\xi_1,\xi_2)\in\mathbb{R}_+^2$. Then up to a constant depending on $\xi$, the Hamiltonian (\ref{c12}) turns into
\begin{align}\begin{split}\label{c23}
H=&\langle\omega^0(\xi),y\rangle+\sum\limits_{j\ge 1\text{ and } j\ne n_1,n_2}\Omega_j(\xi)z_j z_{-j}
+L+\hat{G}+\hat{K}+\hat{T}+W,
\end{split}\end{align}
with
\begin{align}\begin{split}\label{c24}
\omega_1^0(\xi)=&\frac{n_1}{1+n_1^2}\big[1+\frac{n_1^2\xi_1}{2\pi^2(1+n_1^2)^2}+\frac{3n_1 n_2\xi_1^{1/2}\xi_2^{1/2}}{\pi^2(1+n_1^2)(1+n_2^2)}+\frac{3 n_2^2\xi_2}{2\pi^2(1+n_2^2)^2}\big]\\
&+\sum_{j=0}^{4}R_{1j}\xi_1^{\frac{4-j}{2}}\xi_2^{\frac{j}{2}}+\sum_{j=0}^{6}T_{1j}\xi_1^{\frac{6-j}{2}}\xi_2^{\frac{j}{2}},\\
\omega_2^0(\xi)=&\frac{n_2}{1+n_2^2}\big[1+\frac{n_2^2\xi_2}{2\pi^2(1+n_2^2)^2}+\frac{3n_1 n_2\xi_1^{1/2}\xi_2^{1/2}}{\pi^2(1+n_1^2)(1+n_2^2)}+\frac{3n_1^2 \xi_1}{2\pi^2(1+n_1^2)^2}\big]\\
&+\sum_{j=0}^{4}R_{2j}\xi_1^{\frac{4-j}{2}}\xi_2^{\frac{j}{2}}+\sum_{j=0}^{6}T_{2j}\xi_1^{\frac{6-j}{2}}\xi_2^{\frac{j}{2}},\\
\end{split}
\end{align}
where the real coefficients $R_{1j},R_{2j},T_{1j},T_{2j}$ depend only on $R_j$ and $T_j$,
\begin{equation}\label{ee36}
  \Omega_j(\xi)=\frac{j}{1+j^2}\biggl[1+\frac{3n_1^2 \xi_1}{2\pi^2(1+n_1^2)^2}+\frac{3n_2^2 \xi_2}{2\pi^2(1+n_2^2)^2}+\frac{6n_1 n_2\xi_1^{1/2}\xi_2^{1/2}}{\pi^2(1+n_1^2)(1+n_2^2)}\biggr],
\end{equation}
and
\begin{equation}\label{d13}
 L=O(|y|^3+|\xi|^{\frac{1}{2}}|y|^2+|\xi|^{\frac{1}{2}}|y|\|\hat{z}\|_p^2+|y|^2\|\hat{z}\|_p^2).
\end{equation}

\begin{enumerate}
\item [Step 1:] Checking the Assumption A of Theorem \ref{thmA}. By some easy computations, ones have
\begin{align*}
\begin{split}
\frac{\partial\omega_1^0(\xi)}{\partial\xi_1}=&\frac{n_1^3}{2\pi^2(1+n_1^2)^3}+\frac{3n_1^2 n_2\xi_1^{-1/2}\xi_2^{1/2}}{2\pi^2(1+n_1^2)^2(1+n_2^2)}\\
&+\sum_{j=0}^{3}\frac{4-j}{2}R_{1j}\xi_1^{\frac{2-j}{2}}\xi_2^{\frac{j}{2}}+\sum_{j=0}^{5}T_{1j}\frac{6-j}{2}\xi_1^{\frac{4-j}{2}}\xi_2^{\frac{j}{2}},\\
\frac{\partial\omega_1^0(\xi)}{\xi_2}=&\frac{3n_1 n_2^2}{2\pi^2(1+n_1^2)(1+n_2^2)^2}+\frac{3n_1^2 n_2\xi_1^{1/2}\xi_2^{-1/2}}{2\pi^2(1+n_1^2)^2(1+n_2^2)}\\
&+\sum_{j=1}^{4}\frac{j}{2}R_{1j}\xi_1^{\frac{4-j}{2}}\xi_2^{\frac{j-2}{2}}+\sum_{j=1}^{6}T_{1j}\frac{j}{2}\xi_1^{\frac{6-j}{2}}\xi_2^{\frac{j-2}{2}},\\
\end{split}
\end{align*}
\begin{align*}
\begin{split}
\frac{\partial\omega_2^0(\xi)}{\partial\xi_1}=&\frac{3n_1^2 n_2}{2\pi^2(1+n_1^2)^2(1+n_2^2)}+\frac{3n_1 n_2^2\xi_1^{-1/2}\xi_2^{1/2}}{2\pi^2(1+n_1^2)(1+n_2^2)^2}\\
&+\sum_{j=0}^{3}\frac{4-j}{2}R_{2j}\xi_1^{\frac{2-j}{2}}\xi_2^{\frac{j}{2}}+\sum_{j=0}^{5}T_{2j}\frac{6-j}{2}\xi_1^{\frac{4-j}{2}}\xi_2^{\frac{j}{2}},\\
\frac{\partial\omega_2^0(\xi)}{\xi_2}=&\frac{ n_2^3}{2\pi^2(1+n_2^2)^3}+\frac{3n_1 n_2^2\xi_1^{1/2}\xi_2^{-1/2}}{2\pi^2(1+n_1^2)(1+n_2^2)^2}\\
&+\sum_{j=1}^{4}\frac{j}{2}R_{2j}\xi_1^{\frac{4-j}{2}}\xi_2^{\frac{j-2}{2}}+\sum_{j=1}^{6}T_{2j}\frac{j}{2}\xi_1^{\frac{6-j}{2}}\xi_2^{\frac{j-2}{2}}.\\
\end{split}
\end{align*}
Choose the parameter domain in the following way,
\begin{align}
\mathcal{O}_*&=\{\xi=(\xi_1,\xi_2): \epsilon^{\frac{1}{2}}\le \xi_1,\xi_2\le 4\epsilon^{\frac{1}{2}}\},
\end{align}
where $\epsilon$ is a positive small constant.
Thus, we further obtain
\begin{align}\label{c26}\begin{split}
&2\pi^4(1+n_1^2)^2(1+n_2^2)^2\det\partial_{\xi}\omega^0(\xi)\\
=&2\pi^4(1+n_1^2)^2(1+n_2^2)^2\big[\frac{\partial\omega_1^0}{\partial\xi_1}\frac{\partial\omega_2^0}{\partial\xi_2}-\frac{\partial\omega_1^0}{\partial\xi_2}\frac{\partial\omega_2^0}{\partial\xi_1}\big]\\
=&-n_1^2 n_2^2\biggl[\frac{3n_1^2\xi_1^{1/2}\xi_2^{-1/2}}{(1+n_1^2)^2}+\frac{3n_2^2\xi_1^{-1/2}\xi_2^{1/2}}{(1+n_2^2)^2}+\frac{4n_1 n_2}{(1+n_1^2)(1+n_2^2)}\biggr]+o(\epsilon^{\frac{1}{2}})\\
\leq&-\frac{n_1^2 n_2^2}{2}\biggl[\frac{3n_1^2}{2(1+n_1^2)^2}+\frac{3n_2^2}{2(1+n_2^2)^2}+\frac{4n_1 n_2}{(1+n_1^2)(1+n_2^2)}\biggr].
\end{split}\end{align}
Therefore, ones conclude that
\begin{equation*}
  \inf\limits_{\xi\in\mathcal{O}_*}|\det{\partial_{\xi}\omega^0(\xi)}|>c_1,
\end{equation*}
where $c_1$ depends on $n_1,n_2.$ Besides, due to the fact $\lambda_{n_1}>\lambda_{n_2},$ it follows that
\begin{align}\begin{split}
&\sup\limits_{\xi\in\mathcal{O}_*}|\partial_{\xi}\omega^0(\xi)|\\
=&\max\limits_{\xi\in\mathcal{O}_*} \{|\frac{\partial\omega_1^0}{\partial{\xi_1}}|,|\frac{\partial\omega_1^0}{\partial{\xi_2}}|,|\frac{\partial\omega_2^0}{\partial{\xi_1}}|,|\frac{\partial\omega_2^0}{\partial{\xi_2}}|\}
\le \frac{5n_1^3}{\pi^2(1+n_1^2)^3}.
\end{split}\end{align}
Now we have proven that the Assumption A holds true.

\item [Step 2:] Checking the Assumption B of Theorem \ref{thmA}. By (\ref{ee36}), one has
\begin{align}\begin{split}
\frac{\partial\Omega_j(\xi)}{\partial\xi_1}&=\frac{j}{1+j^2}\biggl[\frac{3n_1^2}{2\pi^2(1+n_1^2)^2}+\frac{3n_1 n_2\xi_1^{-1/2}\xi_2^{1/2}}{\pi^2(1+n_1^2)(1+n_2^2)}\biggr],\\
\frac{\partial\Omega_j(\xi)}{\partial\xi_2}&=\frac{j}{1+j^2}\biggl[\frac{3n_2^2}{2\pi^2(1+n_2^2)^2}+\frac{3n_1 n_2\xi_1^{1/2}\xi_2^{-1/2}}{\pi^2(1+n_1^2)(1+n_2^2)}\biggr].\\
\end{split}
\end{align}
Therefore, take $c_{13}=\frac{15n^2_1}{2 \pi^2(1+n_1^2)^2}$, we have
\begin{align}
\sup\limits_{\xi\in\mathcal{O}_*}|\partial_{\xi}\Omega_j(\xi)|\le c_{13}|j|^{-1}.
\end{align}
In view of $n_1\gg 1,$ one obtains that $c_{13}$ is small enough. As mentioned in Remark 2 of \cite{Yuan2018}, when $c_{13}$ is small enough, one automatically has that
\begin{align}
\dfrac{d^*}{d\omega}\biggl[(k,\omega)\pm\Omega_j\biggr]>c^*,\quad \dfrac{d^*}{d\omega}\biggl[(k,\omega)\pm(\Omega_j+\Omega_k)\biggr]>c^*
\end{align}
holds true for some positive constant $c^*>0,$ in which $\omega=\omega^0(\xi).$
Let $c_{11}=\frac{3}{2}$ and $c_{12}=2$, since $1\ll n_1\ll n_2$ and $1/4\le\xi_1/\xi_2\leq4$, we get
\begin{align}
c_{11}|j|^{-1}\le \Omega_j(\xi)\le c_{12}|j|^{-1}.
\end{align}
\item[Step 3:] Verifying the Assumption C of Theorem \ref{thmA}. At first, we choose a neighborhood $D_p(s_0,\epsilon^{\frac{5}{8}})$ in the phase space $\mathcal{P}^p$ (The definitions of $D_p(s_0,\epsilon^{\frac{5}{8}})$ and $\mathcal{P}^p$ refer to (\ref{ee47}) and (\ref{ee48})). From (\ref{c07-3}), (\ref{c22-3}), (\ref{y7}), (\ref{y11}), (\ref{d13}) and (\ref{ee49}), we estimate them and obtain that
\begin{equation*}
  \Arrowvert\lfloor{X_L}\rceil\Arrowvert_{q,D_p(s_0,\epsilon^{\frac{5}{8}})\times\mathcal{O}_*}=O(\epsilon^{\frac{3}{2}}),
\end{equation*}
\begin{equation*}
  \Arrowvert\lfloor{X_{\hat{G}}}\rceil\Arrowvert_{q,D_p(s_0,\epsilon^{\frac{5}{8}})\times\mathcal{O}_*}=O(|\xi|^{\frac{3}{4}}||\hat{z}||^3_p\big/\epsilon^{\frac{5}{4}})=O(\epsilon),
\end{equation*}
\begin{equation*}
  \Arrowvert\lfloor{X_{\hat{R}}}\rceil\Arrowvert_{q,D_p(s_0,\epsilon^{\frac{5}{8}})\times\mathcal{O}_*}=O(|\xi|^{\frac{8}{4}}||\hat{z}||^2_p\big/\epsilon^{\frac{5}{4}})=O(\epsilon),
\end{equation*}
\begin{equation*}
  \Arrowvert\lfloor{X_{\hat{T}}}\rceil\Arrowvert_{q,D_p(s_0,\epsilon^{\frac{5}{8}})\times\mathcal{O}_*}=O(|\xi|^{\frac{13}{4}}||\hat{z}||_p\big/\epsilon^{\frac{5}{4}})=O(\epsilon),
\end{equation*}
\begin{equation*}
  \Arrowvert\lfloor{X_{W}}\rceil\Arrowvert_{q,D_p(s_0,\epsilon^{\frac{5}{8}})\times\mathcal{O}_*}=O(|\xi|^{\frac{18}{4}}\big/\epsilon^{\frac{5}{4}})=O(\epsilon),
\end{equation*}
Let $P=L+\hat{G}+\hat{R}+\hat{T}+W$, and it is easy to check that
\begin{align}
\Arrowvert \lfloor X_{P}\rceil\Arrowvert_{q,D_{p}(s_0,r_0)\times\mathcal{O}_*}\le C\epsilon,\; \Arrowvert \lfloor \partial_{\xi}X_{P}\rceil\Arrowvert_{q,D_p(s_0,r_0)\times\mathcal{O}_*}\le C\epsilon^{\frac{1}{2}}.
\end{align}
\item[Step 4:] Verifying the Assumption D of Theorem \ref{thmA}. In (\ref{c04}), $u$ is real-valued if $z_{-j}$ is the complex conjugate of $z_j$. Thus Assumption D holds true.
\item[Step 5:] Verifying the Assumption E of Theorem \ref{thmA}. In view of $B=0$, Assumption E follows natrually.
	\end{enumerate}
Our main Theorem \ref{t1} follows directly from the Theorem \ref{thmA}.

\section{Appendix}
\subsection{KAM theorem with normal frequencies of finite limit-point}
In this appendix, we shall present the new KAM theorem with normal formal frequencies of finite limit-point developed by the third author \cite{Yuan2018}. To this end, we shall introduce some notations. We start with an infinite dimensional Hamiltonian in the form of an integrable parameter dependent normal form part $N$ plus a Hamiltonian perturbation $P$
\begin{align}
H=N+P=(\omega^0(\xi),y)+\sum\limits_{j\in\mathbb{Z}}\Omega_j(\xi)z_j\bar{z}_j+\langle B^0(\xi)z,\bar{z}\rangle+P(x,y,z,\bar{z};\xi)
\end{align}
where the parameter $\xi\in\mathcal{O}\subset\mathbb{R}^n$, and $(x,y,z,\bar{z})\in \mathbb{T}^n\times\mathbb{R}^n\times h_p\times h_p$.
For $x\in\mathbb{R}^n$, denote the Eulidean norm of $x$ by $|x|$. For given $p>\frac{1}{2},~\kappa>0$, define $q=p+\kappa.$ Given positive integer $n>0,$ let $\mathbb{T}_s^n$ represent the complexization of $\mathbb{T}^n$ with width $s>0,$ that is,
$$\mathbb{T}_s^n:=\{x\in\mathbb{C}^n/(2\pi\mathbb{Z}^n):\quad |\Im x|\le s\}.$$
For given $s_0>0$, define the phase space as follows
\begin{equation}\label{ee47}
  \mathcal{P}^p:=\mathbb{T}_{s_0}^n\times\mathbb{C}^n\times h_p\times h_p.
\end{equation}
When the perturbation term $P$ vanishes, the normal formal part $N$ admits $n$-dimensional torus
$$\mathcal{T}_0^n:=\mathbb{T}^n\times\{y=0\}\times\{z=0\}\times\{\bar{z}=0\}.$$
Let us introduce the neighborhood of the torus $\mathcal{T}_0^n$ in $\mathcal{P}^p$ as follows
\begin{equation}\label{ee48}
  D_p(s_0,r_0):=\{(x,y,z,\bar{z})\in\mathcal{P}^p:|\Im x|<s, |y|<r_0^2,\|z\|_p<r_0,\|\bar{z}\|_p<r_0\}.
\end{equation}

Let us introduce some desired norms. Take $0<s\le s_0$, $0<r<r_0$ and $\mathcal{O}_*\subset\mathcal{O}$, for a map $g(x,\xi):\mathbb{T}_s^n\times\mathcal{O}_*\to\mathbb{C}^n$, define
$$|g|_{s,\mathcal{O}_*}^2=\sup\limits_{\xi\in\mathcal{O}_*}\sum\limits_{k\in\mathbb{Z}^n}|\hat{g}(k,\xi)|^2 e^{2|k|s},$$
in which $\hat{g}(k,\xi)$ represents the $k$-Fourier coefficient of $g(x,\xi)$. For a map $g:\mathbb{T}_s^n\times\mathcal{O}_*\to h_p,$ define its form as follows
$$\|g\|_{p,s,\mathcal{O}_*}=\sup\limits_{\xi\in\mathcal{O}_*}\sum\limits_{k\in\mathbb{Z}^n}\|\hat{g}(k,\xi)\|_p^2 e^{2|k|s}.$$
For a map $g:D_p(s,r)\times\mathcal{O}_*\to h_p,$ define
$$\|g(\cdot,y,z,\bar{z};\xi)\|_{p,s}^2=\sum\limits_{k\in\mathbb{Z}^n} e^{2|k|s}\|\hat{g}(k;y,z,\bar{z};\xi)\|_{p}^2,$$
$$\|g\|_{p,q,s,r,\mathcal{O}_*}=\sup\limits_{\xi\in\mathcal{O}_*,|y|<r^2,\|z\|_p<r,\|\bar{z}\|_p<r}\|g(\cdot,y,z,\bar{z};\xi)\|_{q,s}.$$
Consider a map $g:D_p(s,r)\times\mathcal{O}_*\to\mathbb{C}^n,$ define
$$|g|_{p,s,r,\mathcal{O}_*}=\sup\limits_{\xi\in\mathcal{O}_*,|y|<r^2,\|z\|_p<r,\|\bar{z}\|_p<r}\sqrt{\sum\limits_{k\in\mathbb{Z}^n}|\hat{g}(k;y,z,\bar{z};\xi)|^2 e^{2|k|s}}.$$
For a vector field $$U=(X,Y,Z,\bar{Z}):D_p(s,r)\times\mathcal{O}_*\to \mathcal{P}^q,$$
define its form in the following manner,
\begin{equation}\label{ee49}
  \Arrowvert{U}\Arrowvert_{q,D_p(s,r)\times\mathcal{O}_*}=\sqrt{|X|_{p,s,r,\mathcal{O}_*}^2+|Y|_{p,s,r,\mathcal{O}_*}^2+\Arrowvert{Z}\Arrowvert_{p,q,s,r,\mathcal{O}_*}^2+\Arrowvert{\bar{Z}}\Arrowvert_{p,q,s,r,\mathcal{O}_*}^2}.
\end{equation}
For a scalar complex function defined on $D_{p}(s_0,r_0)\times\mathcal{O}_*,$
$$g(x,y,z,\bar{z};\xi)=\sum\limits_{k\in\mathbb{Z}^n,\gamma\in\mathbb{Z}_+^n,\alpha,\beta\in\mathbb{Z}_+^{\mathbb{Z}}}g_{k,\gamma,\alpha,\beta}(\xi)e^{i(k,x)}y^{\gamma}z^{\alpha}\bar{z}^{\beta},$$
define its modulus as follows:
$$\lfloor g\rceil=\lfloor g(x,y,z,\bar{z};\xi)\rceil=\sum\limits_{k\in\mathbb{Z}^n,\gamma\in\mathbb{Z}_+^n,\alpha,\beta\in\mathbb{Z}_+^{\mathbb{Z}}}|g_{k,\gamma,\alpha,\beta}(\xi)|e^{i(k,x)}y^{\gamma}z^{\alpha}\bar{z}^{\beta}.$$
For an operator or matrix $$g(x,y,z,\bar{z};\xi)=(g_{jk}(x,y,z,\bar{z};\xi)\in\mathbb{C}:j,k\in\mathbb{Z}\text{ or a subset of }\mathbb{Z}),$$
define
$$\lfloor g(x,y,z,\bar{z};\xi)\rceil=(|g_{jk}(x,y,z,\bar{z};\xi)|:j,k\in\mathbb{Z}\text{ or a subset of }\mathbb{Z}).$$

\begin{description}
\item[Assumption A]:(Non-degeneracy.) Suppose that $\omega^0(\xi):\mathcal{O}\subset\mathbb{R}^n\to\mathbb{Z}^n$ is real continuously differentiable in $\xi\in\mathcal{O}$ in the sense of Whitney. Suppose there exist two positive constants $c_1, c_2$ fulfilling
\begin{align}\begin{split}
&\inf\limits_{\xi\in\mathcal{O}}|\det \partial_{\xi}\omega^0(\xi)|\ge c_1,\\
&\sup\limits_{\xi\in\mathcal{O}}|\det\partial_{\xi}\omega^0(\xi)|\le c_2.\\
\end{split}
\end{align}
\item [Assumption B]:(The normal frequencies clustering at the origin.)Suppose $\Omega_j=\Omega_j(\xi)$ are real and continuously differentiable in $\xi\in\mathcal{O}$. Assume that
there are four absolute positive constants $c_{11},c_{12},c_{13},\kappa$ such that
\begin{align}
c_{11}|j|^{-\kappa}\le \Omega_j(\xi)\le c_{12}|j|^{-\kappa},\forall \xi\in\mathcal{O}, j\in\mathbb{Z},
\end{align}
and
\begin{align}
\sup\limits_{\xi\in\mathcal{O}}|\partial_{\xi}\Omega_j(\xi)|\le c_{13}|j|^{-\kappa}, j\in\mathbb{Z}.
\end{align}
Assume further for each $k\in\mathbb{Z}^n\subset\{0\}, j,k\in\mathbb{Z},$
\begin{align}
\frac{d^*}{d\omega}\biggl[(k,\omega)\pm\Omega_j\biggr]>0,\quad \frac{d^*}{d\omega}\biggl[(k,\omega)\pm(\Omega_j+\Omega_k)\biggr]>0,
\end{align}
in which $\Omega_j=\Omega_j(\xi(\omega))=\Omega_j((\omega^0)^{-1}(\omega))$ and $\frac{d^*}{d\omega}$ represents the directional derivative along the direction such that $\frac{d^*}{d\omega}(k,\omega)\ge 0.$
\item[Assumption C]:(Regularity.) Suppose that the perturbation term $P(x,y,z,\bar{z};\xi)$ defined on the domain $D_p(s_0,r_0)\times\mathcal{O}$ is analytic in the spatial coordinate and $C^1-$ smooth in $\xi$ of the parameter $\xi\in\mathcal{O},$ and for every $\xi\in\mathcal{O}$, the modulus $\lfloor X_P\rceil$ of its Hamiltonian vector field
 $X_P=(P_y,-P_x,i\partial_{\bar{z}}P,-i\partial_{z}P)$ determines an analytic map
 $$\lfloor X_P\rceil:D_p(s_0,r_0)\subset\mathcal{P}^p\to\mathcal{P}^q,$$
 satisfying
 $$\Arrowvert \lfloor X_P\rceil\Arrowvert_{q,D_{p}(s_0,r_0)\times\mathcal{O}}\le C\epsilon,\quad \Arrowvert \lfloor \partial_{\xi}X_P\rceil\Arrowvert_{q,D_p(s_0,r_0)\times\mathcal{O}}\le C\epsilon^{\frac{1}{2}}.$$
 \item [Assumption D]:(Reality.) The Hamiltonian functions $N(x,y,z,\bar{z};\xi)$ and $P(x,y,z,\bar{z};\xi)$ are real when $x,y\in\mathbb{R}.$
 \item[Assumption E]: For any $\xi\in\mathcal{O}$, the modulus of the operator $B_0$ is small with respect to $q=p+\kappa$ in the following way:
 $$\sup\limits_{\xi\in\mathcal{O}}\|\lfloor B^0(\xi)\rceil\|_{h_p\to h_q}\le C\epsilon,\quad \sup\limits_{\xi\in\mathcal{O}}\|\lfloor\partial_{\xi}B^0(\xi)\rceil\|_{h_p\to h_q}\le C\epsilon.$$
\end{description}
\begin{theorem}\label{thmA}(\cite{Yuan2018})
Assume that the Hamiltonian $H=N+P$ fulfills the Assumptions (A-E). Then there exists a small $\epsilon^*=\epsilon^*(n,p,q,\mathcal{O})>0$, such that for any $0<\epsilon_0<\epsilon^*,$ there exists a subset $\mathcal{O}_*\subset\mathcal{O}$ and a $\gamma=\gamma(\epsilon)$ satisfying
\begin{align}\label{c31}
\text{Meas}\mathcal{O}_*\ge\text{Meas}{O}(1-O(\gamma(\epsilon))),\quad \lim\limits_{\epsilon_0\to 0}\gamma(\epsilon)=0,
\end{align}
and there also exist a symplectic transformation
\begin{align}\label{c32}
\Phi:D_p(\frac{s_0}{2},\frac{r_0}{2})\times\mathcal{O}_*\subset\mathcal{P}^p\to D_p(s_0,r_0)\times\mathcal{O}\subset\mathcal{P}^q,
\end{align}
such that $H=N+P$ is changed into
\begin{align}\label{c33}
H^\infty=H\circ\Phi=(\omega(\xi),y)+\sum\limits_{j\in\mathbb{Z}}\Omega_j(\xi)z_j\bar{z}_j+\langle B^\infty(\xi)z,\bar{z}\rangle+R,
\end{align}
where
\begin{align}\label{c34}
R=O(|y|^2+|y|\|z\|_p+\|z\|_p^3),
\end{align}
\begin{align}\label{c35}
\|\lfloor X_R\rceil\|_{q,D_p(\frac{s_0}{2},\frac{r_0}{2})\times{\mathcal{O}_*}}\le C\epsilon,\quad \|\lfloor \partial_{\xi}X_R\rceil\|_{q,D_{p}(\frac{s_0}{2},\frac{r_0}{2})\times\mathcal{O_*}}\le C\epsilon^{\frac{1}{2}},
\end{align}
and for any $\xi\in\mathcal{O}_*,$ the operator $B^\infty(\xi)$ fulfills
\begin{align}\label{c36}
\|\lfloor B^\infty(\xi)-B^0(\xi)\rceil\|_{h_p\to h_q}\le C\epsilon,\quad \|\lfloor\partial_{\xi}(B^{\infty}(\xi)-B^0(\xi))\rceil\|_{h_p\to h_q}\le C\epsilon,
\end{align}
and
$\omega:\mathcal{O}_*\to\mathbb{R}^n$ with
\begin{align}\label{c37}
\sup\limits_{\xi\in\mathcal{O}_*}|\omega-\omega^0|\le C\epsilon,\quad \sup\limits_{\xi\in\mathcal{O}_*}|\partial_{\xi}(B^\infty(\xi)-B^0(\xi))|\le C\epsilon^{\frac{1}{2}}.
\end{align}

\end{theorem}
\section*{Acknowlegement}
The authors are grateful to Professor Yuan Xiaoping in Fudan University for his invaluable encouragement and suggestions.

 \begin{center}
           
             \end{center}

\end{document}